\documentclass[12pt]{amsart}
\usepackage{times,amsfonts,amsmath,amstext,amsbsy,amssymb,amsopn,amsthm,upref,eucal,amscd}
\usepackage[T1]{fontenc}
\usepackage{color}
\usepackage{graphicx} 

\numberwithin{equation}{section}

\newtheorem{theorem}{Theorem}

\newtheorem{proposition}{Proposition}

\theoremstyle{definition}

\theoremstyle{remark}
\newtheorem{example}{Example}[section]

\newtheorem{remark}{Remark}[section]

\renewcommand{\emptyset}{\varnothing}

\renewcommand{\epsilon}{\varepsilon}
\renewcommand{\phi}{\varphi}

\newcommand{\SL}{\operatorname{SL}(2,{\mathbb R})}

\newlength{\halfbls}

\begin{document}
\title[Lyapunov exponents of Teichm\"uller curves]{A coding-free simplicity criterion for the Lyapunov exponents of Teichm\"uller curves}
\author{Alex Eskin}
\address{Department of Mathematics, University of Chicago, Chicago, Illinois 60637, USA}
\email{eskin\@@\,math.uchicago.edu}
\urladdr{http://www.math.uchicago.edu/$\sim$eskin}
\author{Carlos Matheus}
\address{Universit\'e Paris 13, Sorbonne Paris Cit\'e, LAGA, CNRS (UMR 7539), F-93430, Villetaneuse, France}
\email{matheus\@@\,impa.br}
\urladdr{http://www.impa.br/$\sim$cmateus}

\date{\today}

\begin{abstract}
In this note we show that the results of H. Furstenberg on the Poisson boundary of lattices of semisimple Lie groups allow to deduce simplicity properties of the Lyapunov spectrum of the Kontsevich-Zorich cocycle of Teichm\"uller curves in moduli spaces of Abelian differentials without the usage of codings of the Teichm\"uller flow. As an application, we show the simplicity of some Lyapunov exponents in the setting of (some) Prym Teichm\"uller curves of genus $4$ where a coding-based approach seems hard to implement because of the poor knowledge of the Veech group of these Teichm\"uller curves. Finally, we extend the discussion in this note to show the simplicity of Lyapunov exponents coming from (high weight) variations of Hodge structures associated to mirror quintic Calabi-Yau threefolds.
\end{abstract}

\maketitle

\setcounter{tocdepth}{1}
\tableofcontents

\section{Introduction}\label{s.introduction}

After the seminal works of A. Zorich~\cite{Z94} and M. Kontsevich~\cite{K},  it is known that the qualitative and quantitative features of the Lyapunov exponents of the Kontsevich-Zorich cocycle over the Teichm\"uller flow on the moduli space of unit area Abelian and quadratic differentials inspire several applications to different subjects such as: 
\begin{itemize}
\item the description of the deviation of ergodic averages of interval exchange transformations (\cite{Z94}, \cite{Fo02}), 
\item the computation of the diffusion rates of trajectories in the so-called Ehrenfest wind-tree model for Lorenz gases \cite{DHL},
\item the distinction between commensurability classes of all presently known non-arithmetic ball quotients \cite{Kappes:Moeller}, etc.
\end{itemize}

Concerning the qualitative properties of the Lyapunov spectrum of the Kontsevich-Zorich cocycle with respect to Teichm\"uller flow invariant probabilities, the so-called \emph{simplicity} property, that is, Lyapunov exponents have multiplicity $1$, is one of the most prominent because, from the theoretical point of view, it gives the most complete picture of the fiber-dynamics of the Kontsevich-Zorich cocycle. Indeed, by Oseledets theorem, the simplicity property implies that the action of the Kontsevich-Zorich cocycle on almost every fiber can be completely diagonalized, and, for example, this is a important information in the applications to deviations of ergodic averages of interval exchange transformations. 

To the best of the authors' knowledge, currently there exist two results for the simplicity of the Lyapunov spectrum of the Kontsevich-Zorich cocycle. The first one is the celebrated theorem of A. Avila and M. Viana~\cite{AV} (proving the former Kontsevich-Zorich conjecture) showing that the simplicity of Lyapunov exponents of the Kontsevich-Zorich cocycle with respect to the so-called \emph{Masur-Veech} probability measures. The second one is a recent result by the second author, M. M\"oller and J.-C. Yoccoz~\cite{MMY} giving a criterion (inspired by the techniques of~\cite{AV}) for the simplicity of Lyapunov exponents of invariant probability measures associated to the so-called \emph{arithmetic Teichm\"uller curves}.  

A common technical feature of the proofs of simplicity in the papers~\cite{AV} and~\cite{MMY} is the fact that they somehow depend on particular \emph{codings} (countable Markov partitions) for the Teichm\"uller geodesic flow that are well-adapted to the invariant probability measures at hand. In particular, A. Avila and M. Viana~\cite{AV} use the so-called Rauzy-Veech algorithm for their coding, and the second author, M. M\"oller and J.-C. Yoccoz~\cite{MMY} use certain finite extensions of the continued fraction algorithm for their coding. Of course, the usage of a coding normally makes life technically easier, but it has a drawback coming from the fact that codings adapted to \emph{arbitrary} invariant measures are not easy to produce in general. 

The goal of this note is the description of a ``coding-free''
simplicity criterion for invariant probabilities associated to
Teichm\"uller curves (i.e., closed $\SL$-orbits). Roughly speaking,
given a Teichm\"uller curve $\mathcal{T}$, we show that the knowledge
of the action on homology of affine diffeomorphisms of the translation
surfaces in $\mathcal{T}$ suffices to show the simplicity of the
Lyapunov spectrum \emph{without} using any coding for the geodesic
flow on $\mathcal{T}$. Here, the basic idea allowing us to forget
about codings is a theorem of H. Furstenberg~\cite{Fu71} ensuring
that, given a lattice $\Gamma$ in $\SL$, it is possible to find a
probability $\mu$ on $\Gamma$ with $\textrm{supp}(\mu)=\Gamma$ (i.e.,
$\mu$ gives non-zero weights to all elements of $\Gamma$) such that
$(\textrm{SO}(2,\mathbb{R}),\textrm{Leb})$ is the Poisson boundary of
$(\Gamma,\nu)$. This strategy was suggested to the
  first author by A.~Furman many years ago.

We organize this note as follows. In the next section, we quickly review some relevant definitions (moduli spaces of Abelian differentials, Teichm\"uller flow and $\SL$-action on these moduli spaces, Kontsevich-Zorich cocycle, Teichm\"uller curves, etc.), we state our main ``coding-free'' simplicity criterion (see Theorem~\ref{t.Lyapunov}), and we give the (short) proof of our simplicity criterion essentially by combining the results of H. Furstenberg~\cite{Fu71} on one hand, and Y. Guivarch and A. Raugi~\cite{GR,GR2}, I. Goldsheid and G. Margulis~\cite{GM}, and A. Avila and M. Viana~\cite{AV}. Then, we make an application of the main criterion to a class of Prym Teichm\"uller curves (see Theorem~\ref{t.Prym}) where coding-based approaches are hard to perform (because the Veech groups of these Teichm\"uller curves are essentially unknown). Finally, in the last section we show that our main criterion (in Theorem~\ref{t.Lyapunov}) also applies to variations of Hodge structures (of weight two) associated to mirror quintic Calabi-Yau threefolds. 

\begin{remark}\label{r.quadratic-diffs} The natural analog of Theorem~\ref{t.Lyapunov} below for Teichm\"uller curves in the moduli spaces of non-orientable quadratic differentials is also true. Indeed, our proof of Theorem~\ref{t.Lyapunov} can be easily adapted to the case of non-orientable quadratic differentials because these objects are naturally related to Abelian differentials via the usual canonical double-cover construction. However, we prefer to leave the detailed verification of this fact to the reader since we will not need it for our current applications of the ``coding-free'' simplicity criterion. 
\end{remark}

\section{Lyapunov exponents of Teichm\"uller curves}\label{s.Lyapunov}

\subsection{Moduli spaces, Teichm\"uller flow and Kontsevich-Zorich cocycle} The basic references for the facts stated without proofs in this subsection are the surveys \cite{Fo06} of G. Forni and \cite{Z06} of A. Zorich, and the references therein. 

We denote by $\mathcal{H}_g$ the moduli space of Abelian differentials on Riemann surfaces of genus $g\geq 1$, that is, $\mathcal{H}_g$ is the set of equivalence classes of pairs $(M,\omega)$ where $M$ is a compact Riemann surface of genus $g\geq 1$ and $\omega$ is a non-trivial Abelian differential (non-vanishing holomorphic $1$-form) on $M$ \emph{modulo} the action of the group $\textrm{Diff}^+(M)$ of orientation-preserving homeomorphisms of $M$.

In the literature, the pairs $(M,\omega)$ are called \emph{translation surface structures} because, by local integration of $\omega$ outside the set $\Sigma$ of its zeroes, one gets an atlas  on $M-\Sigma$ such that the change of coordinates is given by translations. We refer to this atlas as a translation atlas. 

The translation atlases are useful structures because they make clear that the group $\textrm{GL}^+(2,\mathbb{R})$ acts on $\mathcal{H}_g$ by post-composition with charts of translation atlases. As it turns out, from the point of view of ergodic theory, it is more efficient to study the $\textrm{GL}^+(2,\mathbb{R})$-action on $\mathcal{H}_g$ by focusing on the ``unit hyperboloid'' $\mathcal{H}_g^{(1)}$ of \emph{unit area} translation surfaces in $\mathcal{H}_g$, that is, the translation surfaces $(M,\omega)\in\mathcal{H}_g$ such that the area function $A(M,\omega):=(i/2)\int\omega\wedge\overline{\omega}$ equals $1$. In fact, despite the fact that this ``unit hyperboloid'' is not very far from $\mathcal{H}_g$ (one has to rescale the area of $(M,\omega)$ by multiplying $\omega$ by a constant to get in $\mathcal{H}_g^{(1)}$), it is preserved by the action of the subgroup $\SL$ of $\textrm{GL}^+(2,\mathbb{R})$, and, furthermore, it is possible to show that $\mathcal{H}_g^{(1)}$ supports plenty of $\SL$-invariant \emph{probability} measures. In other words, 
$\mathcal{H}_g^{(1)}$ is the correct object of study as far as ergodic-theoretical methods are concerned. 

In what follows, we'll need the following two basic facts about $\mathcal{H}_g^{(1)}$ and the corresponding $\SL$-action. Firstly, $\mathcal{H}_g^{(1)}$ is stratified into complex orbifolds $\mathcal{H}^{(1)}(k_1,\dots, k_s)$ of dimension $2g+s-1$ obtained by collecting Abelian differentials 
$(M,\omega)\in\mathcal{H}_g^{(1)}$ such that the list of orders of its zeroes is $\kappa=(k_1,\dots,k_s)$, $\sum k_m=2g-2$. Secondly, the $\SL$-action on $\mathcal{H}_g^{(1)}$ preserves each stratum $\mathcal{H}^{(1)}(\kappa)$.

In this language, the \emph{Teichm\"uller flow} $g_t$ on a stratum $\mathcal{H}_g^{(1)}(\kappa)$ of the moduli space $\mathcal{H}_g^{(1)}$ is simply the action of the diagonal subgroup $g_t:=\textrm{diag}(e^t, e^{-t})$ of $\SL$. In their seminal works, M. Kontsevich~\cite{K} and A. Zorich~\cite{Z99} introduced the so-called \emph{Kontsevich-Zorich cocyle} over the Teichm\"uller flow, a fundamental object with applications to the dynamics of interval exchange transformations, translation flows and billiards. The definition of this cocycle is the following. Firstly, one considers the \emph{Teichm\"uller space} $\widehat{\mathcal{H}_g}$ of Abelian differentials, that is, the set of equivalence classes of Abelian differentials $(M,\omega)$ modulo the action of the group $\textrm{Diff}_0^+(M)$ of orientation-preserving homeomorphisms of $M$ isotopic to the identity. Note that the moduli space $\mathcal{H}_g$ is the quotient $\mathcal{H}_g=\widehat{\mathcal{H}_g}/\Gamma_g$ of the Teichm\"uller space by the action of the \emph{mapping-class group} $\Gamma_g:=\textrm{Diff}^+(M)/\textrm{Diff}_0^+(M)$. The \emph{Hodge bundle}
$$H_g^1=\widehat{H_g^1}/\Gamma_g$$
over the moduli space $\mathcal{H}_g$ is the quotient of the trivial bundle $\widehat{H^1_g}=\widehat{\mathcal{H}_g}\times H^1(M,\mathbb{R})$ 
over the Teichm\"uller space by the diagonal action of the mapping-class group $\Gamma_g$ on both factors of $\widehat{H_g^1}$. In this language, the Kontsevich-Zorich cocycle $G_t^{KZ}$ is the quotient $G_t^{KZ}=\widehat{G_t^{KZ}}/\Gamma_g$ of the trivial cocycle $$\widehat{G_t^{KZ}}(\omega,[c])=(g_t(\omega),[c])$$
on $\widehat{H_g^1}$ by $\Gamma_g$.

Concerning the general features of the Lyapunov spectrum, we observe that the Kontsevich-Zorich cocycle $G_t^{KZ}$ preserves the natural (symplectic) intersection form on the fibers $H^1(M,\mathbb{R})$ of the Hodge bundle $H_g^1$. Since $H^1(M,\mathbb{R})$ is a real vector space of dimension $2g$, we have that $G_t^{KZ}$ is a symplectic cocycle and, thus, its Lyapunov spectrum with respect to any Teichm\"uller flow invariant probability $\mu$ has the form
$$\lambda_1^{\mu}\geq\lambda_2^{\mu}\geq\dots\lambda_g^{\mu}\geq-\lambda_g^{\mu}\geq\dots\geq-\lambda_2^{\mu}\geq-\lambda_1^{\mu}$$

Also, it is not hard to check that the Hodge bundle has a decomposition
$$H_g^1=H_{st}^1(M,\mathbb{R})\oplus H_{(0)}^1(M,\mathbb{R})$$
where $H_{st}^1(M,\mathbb{R})$ is the ``tautological'' 
  subbundle generated by the cohomology classes coming from the real
and imaginary parts of $\omega$ and $H_{(0)}^{1}(M,\mathbb{R})$ is the
symplectic orthogonal of $H_{st}^1(M,\mathbb{R})$. Furthermore, this
decomposition is invariant under the $SL(2,\mathbb{R})$ action and the Kontsevich-Zorich cocycle. 

A simple geometrical argument allows one to show that the $2$-dimensional (symplectic) tautological subbundle $H_{st}^1(M,\mathbb{R})$ contributes with the ``tautological'' Lyapunov exponents $\lambda_1^{\mu}=1$ and $-\lambda_1^{\mu}=-1$. Also, G. Forni \cite{Fo02} proved the ``spectral gap'' estimate $\lambda_2^{\mu}<1$. In particular, one has that the Lyapunov spectrum of $G_t^{KZ}$ with respect to any $g_t$-invariant probability $\mu$ has the form 
$$1>\lambda_2^{\mu}\geq\dots\geq\lambda_g^{\mu}\geq-\lambda_g^{\mu}\geq\dots\geq-\lambda_2^{\mu}>-1$$
where the Lyapunov exponents $\pm\lambda_i^{\mu}$, $i=2,\dots,g$ come from the restriction of $G_t^{KZ}$ to the subbundle $H_{(0)}^1(M,\mathbb{R})$. 

In the sequel, we'll be interested in the qualitative features of $\pm\lambda_i^{\mu}$, $i=2,\dots,g$, when $\mu$ is a $\SL$-invariant probability measure whose support is the smallest possible, namely, a single closed $\SL$-orbit, i.e., a \emph{Teichm\"uller curve}.

\subsection{Veech surfaces, Teichm\"uller curves and affine diffeomorphisms}
The basic reference for this subsection is the survey~\cite{HS} of P. Hubert and T. Schmidt.

Let $(M,\omega)\in\mathcal{H}_g^{(1)}$ be a translation surface. We denote by $\textrm{SL}(M,\omega)$ the stabilizer of $(M,\omega)$ with respect to the $\SL$-action on 
$\mathcal{H}_g^{(1)}$. In the literature, $\textrm{SL}(M,\omega)$ is called \emph{Veech group} of $(M,\omega)$. 

By a result of J. Smillie (see e.g. \cite{SW}), it is known that the $\SL$-orbit of $(M,\omega)$ is \emph{closed} in the moduli space if and only if the Veech group $\textrm{SL}(M,\omega)$ is a \emph{lattice} in $\SL$. In this case, $(M,\omega)$ is called a \emph{Veech surface} and its $\SL$-orbit in the moduli space is isomorphic to the unit cotangent bundle $\SL/\textrm{SL}(M,\omega)$ of the (finite-area) hyperbolic surface $\mathbb{H}/\textrm{SL}(M,\omega)$. For this reason, the closed $\SL$-orbits are called \emph{Teichm\"uller curves}.

By definition, the $\SL$-orbit of a Veech surface (i.e., a
Teichm\"uller curve) supports an unique (ergodic) $\SL$-invariant
probability measure. So, in the sequel, this measure will be always understood when we talk about the invariant measure associated to a given Teichm\"uller curve. 

In general, the Veech group $\textrm{SL}(M,\omega)$ is part of an
short exact sequence
$$1\to\textrm{Aut}(M,\omega)\to \textrm{SL}(M,\omega)\to \textrm{Aff}(M,\omega)\to 1$$
where $\textrm{Aut}(M,\omega)$, resp. $\textrm{Aff}(M,\omega)$, is the \emph{automorphism}, resp. \emph{affine}, group of diffeomorphisms of $(M,\omega)$, that is, the set of orientation-preserving homeomorphisms of $(M,\omega)$ preserving the set of zeroes of $\omega$ whose expression in the translation atlas of $(M,\omega)$ is a \emph{translation}, resp. \emph{affine}. 

In particular, \emph{when} $\textrm{Aut}(M,\omega)=\{\textrm{id}\}$, i.e., $(M,\omega)$ has no (non-trivial) automorphisms, we have that $\textrm{SL}(M,\omega)$ is \emph{isomorphic} to 
$\textrm{Aff}(M,\omega)$. This elementary fact has the following interesting consequence for the study of the Kontsevich-Zorich cocycle over Teichm\"uller curves. 

Firstly, by using the duality between of the homology and cohomology, we can think of $G_t^{KZ}$ as the quotient of the trivial cocycle 
$$(\omega,\gamma)\mapsto (g_t(\omega),\gamma)$$
on $\widehat{H_g^1}\times H_1(M,\mathbb{R})$ by the action of the mapping-class group $\Gamma_g$. Secondly, it is possible to show that, in genus $g\geq 2$, the affine group $\textrm{Aff}(M,\omega)$ injects into the mapping-class group $\Gamma_g$. Finally, one has that the stabilizer of the $\SL$-orbit of 
$(M,\omega)$ in $\Gamma_g$ is exactly $\textrm{Aff}(M,\omega)$. Therefore, when $\textrm{Aut}(M,\omega)=\{\textrm{id}\}$, the Kontsevich-Zorich cocycle $G_t^{KZ}$ over a closed $\SL$-orbit can be thought as the quotient of the trivial cocycle 
$$\SL\times H_1(M,\mathbb{R})\to\SL\times H_1(M,\mathbb{R})$$
by the natural action of $\textrm{SL}(M,\omega)$ (on the first factor it is the usual action and in the second factor the Veech group acts via the injections $\textrm{SL}(M,\omega)\to\textrm{Aff}(M,\omega)\to\Gamma_g$). 

In summary, the Kontsevich-Zorich cocycle over a Teichm\"uller curve is closely related to the action on homology of the affine group $\textrm{Aff}(M,\omega)$. In our subsequent discussion, we will systematically adopt this point of view on the Kontsevich-Zorich cocycle. 

Note that, by Poincar\'e duality, $G_t^{KZ}$ leaves invariant a decomposition of the (dual of the) Hodge bundle
$$H_1(M,\mathbb{R})=H_1^{st}(M,\mathbb{R})\oplus H_1^{(0)}(M,\mathbb{R})$$
where the tautological $2$-dimensional symplectic subbundle $H_1^{st}(M,\mathbb{R})$ carries the tautological Lyapunov exponents $\pm1$ and $H_1^{(0)}(M,\mathbb{R})$ is the symplectic orthogonal of $H_1^{st}(M,\mathbb{R})$. 

In the case of Teichm\"uller curves, Deligne's semisimplicity
theorem~\cite{Deligne} (see also~\cite{EMi} and~\cite{Fi}) says that there is an unique (possibly trivial) decomposition 
$$H_1^{(0)}(M,\mathbb{R})=H(1)\oplus\dots\oplus H(k)$$
where $H(i)$ are distinct isotypical components of $\SL$-invariant irreducible subbundles (i.e., $H(i)$ is a direct sum of isomorphic $\SL$-invariant irreducible representations). In this context, it is natural to study the Kontsevich-Zorich cocycle and the action of $\textrm{Aff}(M,\omega)$ restricted to each ``block'' $H(i)$. 

After setting up this notation, we are ready to state and proof our main ``coding-free'' simplicity criterion.

\subsection{A ``coding-free'' simplicity criterion for Lyapunov exponents}
Let $\mathcal{T}$ be a Teichm\"uller curve given by the $\SL$-orbit of a Veech surface $(M,\omega)$ without non-trivial automorphisms, and let $H(i)$ be a block consisting of a symplectic isotypical component of $\SL$-invariant irreducible of the (dual of the) Hodge bundle over 
$\mathcal{T}$. Denote by $\mathcal{G}_i$ the group of symplectic matrices obtained by restricting the action on homology of $\textrm{Aff}(M,\omega)$ to the block $H_i$ of dimension $2r(i)$ (say), and let $\lambda_1(i)\geq\dots\geq\lambda_{r(i)}(i)\geq-\lambda_{r(i)}(i)\geq\dots\geq-\lambda_1(i)$ be the Lyapunov exponents of the restriction of the Kontsevich-Zorich cocycle to $H(i)$ (with respect to the unique $\SL$-invariant probability supported on $\mathcal{T}$). 

For the statement of the simplicity criterion, we need to recall the following concepts coming from the papers~\cite{GR,GR2},~\cite{GM} and~\cite{AV}. Given a monoid $\mathcal{G}$ of matrices acting on a real vector space $H$, we say that 
\begin{itemize}
\item $\mathcal{G}$ is \emph{strongly irreducible} if it doesn't preserve a finite collection of proper subspaces of $H$;
\item $\mathcal{G}$ has the \emph{contraction property} whenever given any probability measure $\mu$ on $\mathbb{P}(H)$ such that $\mu(P)=0$ for any proper projective subspace $P$, one can find a sequence of elements $g_n\in\mathcal{G}$ such that $g_n\mu$ converges to a Dirac mass at some point of $\mathbb{P}(H)$.
\end{itemize} 
Also, if $\mathcal{G}$ is a monoid of \emph{symplectic} matrices on a real symplectic vector space $H$ of dimension $2d$, we say that 
\begin{itemize}
\item $\mathcal{G}$ is \emph{pinching} if $\mathcal{G}$ contains a \emph{pinching matrix} $A$, i.e., a matrix whose eigenvalues are real with distinct norms; 
\item $\mathcal{G}$ is \emph{twisting} if there are a \emph{pinching matrix} $A\in\mathcal{G}$ and a twisting matrix $B\in\mathcal{G}$ with respect to $A$ in the sense that $B(F)\cap F'=\{0\}$ whenever $F$ is an \emph{isotropic} $A$-invariant subspace of dimension $1\leq k\leq d$ and $F'$ is a \emph{co-isotropic} $A$-invariant subspace of dimension $2d-k$.
\end{itemize}

\begin{theorem}\label{t.Lyapunov} In the setting of the previous paragraphs, one has that: 
\begin{itemize}
\item[(a)] if $\mathcal{G}_i$ has the contraction property and $\mathcal{G}_i$ is strongly irreducible, then the ``top'' Lyapunov exponent $\lambda_1(i)$ of the block $H(i)$ is simple (i.e., $\lambda_1(i)>\lambda_2(i)$); 
\item[(b)] if the Zariski closure of $\mathcal{G}_i$ coincides with the full symplectic group $\textrm{Sp}(H)$, then the Lyapunov spectrum of the restriction of $G_t^{KZ}$ to $H(i)$ is simple;
\item[(c)] if $\mathcal{G}_i$ is pinching and twisting, then the Lyapunov spectrum of the restriction of $G_t^{KZ}$ to $H(i)$ is simple.
\end{itemize}
\end{theorem}

As we mentioned in the introduction, the proof of this result is not
very long due to the profound results of H. Furstenberg~\cite{Fu71},
Y. Guivarc'h and A. Raugi~\cite{GR}, I. Goldsheid and
G. Margulis~\cite{GM}, and A. Avila and M. Viana~\cite{AV} (see also
\cite[Lemme 3.9]{GR2}).

\begin{proof} By the results of H. Furstenberg~\cite{Fu71}, we have a
  probability measure $\nu$ on the Veech group $\textrm{SL}(M,\omega)$
  with $\textrm{supp}(\nu)=\textrm{SL}(M,\omega)$ (i.e., $\nu$ gives
  non-zero weights to \emph{all} elements of $\textrm{SL}(M,\omega)$)
  such that the \emph{Poisson boundary}\footnote{See, e.g., the survey
    ~\cite{Furman} of A. Furman for a gentle introduction to the
    Poisson boundary.} of $(\textrm{SL}(M,\omega),\nu)$ is
  $(\textrm{SO}(2,\mathbb{R}),\textrm{Leb})$. By the
    Osceledec multiplicative ergodic theorem applied to random walks
    in $\mathbb{H}$, a typical trajectory
    of the random walk in $\textrm{SL}(M,\omega) i \subset 
\mathbb{H}$  tracks a geodesic ray 
    in  $\mathbb{H}$ up to sublinear error (see, e.g., the (short) proof of Lemma 4.1 of~\cite{CE} where this is explained in more detail). In other words, for almost
    all sequences  $\gamma_i \in \textrm{SL}(M,\omega)$ (with respect
    to the measure $\nu \times \nu \dots \times \nu \dots$) there
    exists a geodesic ray $\{\alpha(t): t\in\mathbb{R}\}\subset\mathbb{H}$ such that  
\begin{displaymath}
\textrm{dist}_{\mathbb{H}}(\gamma_n \dots \gamma_1 \cdot i, \alpha(n)) = o(n). 
\end{displaymath}
We may write $\alpha(t)$ as $g_t(r_\theta \omega)$, where $r_\theta
= \begin{pmatrix} \cos \theta & \sin \theta \\ -\sin \theta & \cos
  \theta \end{pmatrix}$ and as above $g_t$ denotes the geodesic flow,
i.e. the action of $\begin{pmatrix} e^t & 0 \\ 0 &
  e^{-t} \end{pmatrix}$.  

By combining this with Forni's estimate 
  $$\frac{d}{dt}\log\|G_t^{KZ}(\omega)\|\leq 1$$
for all $\omega\in\mathcal{H}_g^{(1)}$ from \cite[\S{2}]{Fo02} (where $\|.\|$ denotes the so-called \emph{Hodge norm} on the Hodge bundle, see~\cite{Fo02}), we deduce that, for almost all sequences $\gamma_i \in
\textrm{SL}(M,\omega)$, there exists $\theta \in [0,2\pi)$ such that
\begin{eqnarray}\label{e.Forni-sublinear-error}
\log\|\rho(\gamma_n \dots \gamma_1) G^{KZ}_{-n}(r_\theta \omega))\| 
&\leq& \textrm{dist}_{\mathbb{H}}(\gamma_n\dots\gamma_1\cdot i, g_n(r_{\theta}\omega)) \\ &=& o(n) \nonumber,
\end{eqnarray}
where $\rho: SL(M,\omega) \to \mathrm{Sp}(2g,\mathbb{Z})$ 
denotes the action on the homology $H_1(M,\mathbb{R})$. Furthermore,
since the Poisson boundary of $(\textrm{SL}(M,\omega),\nu)$ is
$(\textrm{SO}(2,\mathbb{R}),\textrm{Leb})$, the probability
distribution of the ``angle'' $\theta$ is uniform with respect to
the Lebesgue measure on $[0,2\pi)$. (Any distribution
absolutely continuous with respect to Lebesgue measure would be
sufficient here.)

Hence, it follows that the Lyapunov exponents of
the Kontsevich-Zorich cocyle $G_t^{KZ}$ coincide with the Lyapunov
exponents of the random walk on (a subgroup of) $Sp(2g,\mathbb{Z})$
whose defining law is $\rho(\nu)$. Indeed, given a non-trivial element $v\in H_1(M,\mathbb{R})$ and denoting by $\lambda_{\rho(\nu)}(v)$ and $\lambda_{KZ}(v)$ the Lyapunov exponents of $v$ with respect to the random walk of law $\rho(\nu)$ and the Kontsevich-Zorich cocycle (resp.), we see from \eqref{e.Forni-sublinear-error} (and the definition of Lyapunov exponent) that
$$0=\lim\limits_{n\to\infty}\frac{1}{n}\log\frac{\|\rho(\gamma_n)\dots\rho(\gamma_1)v\|}{\|G_n^{KZ}(r_{\theta}\omega)v\|} = \lambda_{\rho(\nu)}(v)-\lambda_{KZ}(v)$$
Here, it is worth to point out that Forni's estimate played an important role because it allowed to transfer the fact (coming from Oseledets theorem) that random products of law $\nu$ track certain geodesic rays in the hyperbolic plane into the estimate \eqref{e.Forni-sublinear-error} that random products with law $\rho(\nu)$ ``track'' certain matrices of the Kontsevich-Zorich cocycle.

In particular, the Lyapunov exponents $$\lambda_1(i)\geq\dots\geq\lambda_{r(i)}(i)\geq-\lambda_{r(i)}(i)\geq\dots\geq-\lambda_1(i)$$ of the restriction of $G_t^{KZ}$ to the block $H(i)$ coincide with the Lyapunov exponents of random products with law $\rho(\nu)$ of the matrices of $\mathcal{G}_i$ (coming from the restriction to $H(i)$ of the action on homology\footnote{Recall that, by hypothesis, $(M,\omega)$ has no non-trivial automorphisms so that $\textrm{SL}(M,\omega)$ injects into $\textrm{Aff}(M,\omega)$.} of the elements of 
$\textrm{SL}(M,\omega)$). 

At this point, the proof of Theorem~\ref{t.Lyapunov} is complete, since the analogous results for products of 
independent identically distributed matrices are well known. 
Indeed, the item (a) follows from the results of Y.  Guivarc'h and
 A. Raugi~\cite{GR} saying that the top Lyapunov exponents of random
 products of elements of a monoid $\mathcal{G}$ is simple whenever
 $\mathcal{G}$ has the contraction property and $\mathcal{G}$ is
 strongly irreducible. Similarly, the item (b), resp. (c), follows
 from the results of I. Goldsheid and G. Margulis~\cite{GM},
 resp. A. Avila and M. Viana~\cite{AV} (see also
 \cite[Lemme~3.9]{GR2}), saying that the Lyapunov
 spectrum of random products of elements of a monoid $\mathcal{G}$ is
 simple whenever the Zariski closure of $\mathcal{G}$ coincides with
 the full symplectic group, resp. $\mathcal{G}$ is pinching and
 twisting.
\end{proof}

\begin{remark} The main reason for the statement of
Theorem~\ref{t.Lyapunov} to focus on Teichm\"uller curves comes from
the fact that Furstenberg's theorem~\cite{Fu71} concerns
homogenous spaces. In particular, it is an interesting open
question to know whether an analog of Furstenberg's theorem holds in
the non-homogenous setting of supports of general $\SL$-invariant
measures in moduli spaces of Abelian differentials.
\end{remark}

\section{Application to some Prym curves of genus 4}\label{s.Prym}

We start this section with a quick review of some features of the so-called \emph{Prym eigenforms} and their $\SL$-orbits. For more detailed expositions on this subject, see \cite{McM06} and \cite{LN}.

\subsection{Prym eigenforms and Weierstrass loci} Let $X$ be a Riemann surface equipped with a holomorphic involution $\rho:X\to X$ whose Prym variety
$$\textrm{Prym}(X,\rho):=\Omega^-(X)^*/H_1^-(X,\mathbb{Z})$$
is an Abelian variety of complex dimension 2 admitting real multiplication by the real order $\mathcal{O}_D$ of discriminant $D\in\mathbb{N}$, $D\equiv 0$ or $1$ (mod $4$). Here, $\Omega^-(X)$, resp. $H_1^-(X,\mathbb{Z})$, is the subset of $\rho$-anti-invariant holomorphic $1$-forms, resp., integral homology cycles, on
$X$. In this context, we say that $\omega\in\Omega^-(X) -\{0\}$ is a \emph{Prym eigenform} whenever $\mathcal{O}_D\cdot\omega\subset\mathbb{C}\omega$.

By a simple application of the Riemann-Hurwitz formula (see Theorem 3.1 of \cite{McM06} or Remark 2.4 of \cite{LN}), it is not hard to see that Prym eigenforms exist only when $X$ has genus $2\leq g(X)\leq 5$.

We denote by $\Omega E_D(k_1,\dots, k_s)$ the subset of Prym eigenforms with multiplication by $\mathcal{O}_D$ in a stratum $\mathcal{H}^{(1)}(k_1,\dots, k_s)$ of the moduli space $\mathcal{H}^{(1)}_g$.

The Prym eigenforms are interesting objects because, as it was shown by C. McMullen \cite{McM06}, $\Omega E_D(k_1,\dots, k_s)$ are closed $\SL$-invariant subsets of the stratum $\mathcal{H}(k_1,\dots, k_s)$.

For our current purposes, we will focus on the so-called \emph{Weierstrass loci} $\Omega E_D(2g-2)$, that is, the loci of Prym eigenforms in the minimal stratum $\mathcal{H}^{(1)}(2g-2)$. Again, by Riemann-Hurwitz formula (see, e.g., Remark 2.4 of \cite{LN}), one has that, for genus $g=5$, $\Omega E_D(8)=\emptyset$, that is, the Weierstrass loci are interesting only for $g=2, 3$ and $4$. Also, by a certain ``dimension counting'' argument, C. McMullen was also able to show that, for $g=2, 3, 4$, the Weierstrass loci $\Omega E_D(2g-2)$ consist of the union of finitely many Teichm\"uller curves.

In genus $2$, the number of connected components of the Weierstrass loci $\Omega E_D(2)$ was completely determined by C. McMullen \cite{McM05}. More recently, E. Lanneau and D.-M. Nguyen \cite{LN} completely determined the number of connected components of $\Omega E_D(4)$ (Weierstrass loci in genus $3$) and almost completely determined the number of connected components of $\Omega E_D(6)$ (Weierstrass loci in genus $4$). Very roughly speaking, the basic strategy (introduced by C. McMullen in \cite{McM05}) to compute the number of connected components of $\Omega E_D(2g-2)$ is the following: firstly, one produces \emph{prototypes} for cylinder decompositions along periodic directions of translation surfaces (Abelian differentials) in $\Omega E_D(2g-2)$, that is, canonical representatives for the cusps of the Teichm\"uller curves inside
$\Omega E_D(2g-2)$; secondly, one uses \emph{butterfly moves} to connect prototypes and eventually determine the number of connected components of $\Omega E_D(2g-2)$.

For later use, we depict below Models A and B prototypes of E. Lanneau and D.-M. Nguyen (see Figures 18 and 19 of \cite{LN}) for cylinder decompositions along periodic directions of translation surface in $\Omega E_D(6)$ (i.e., the Weierstrass loci in genus $4$).

\begin{figure}[htb!]
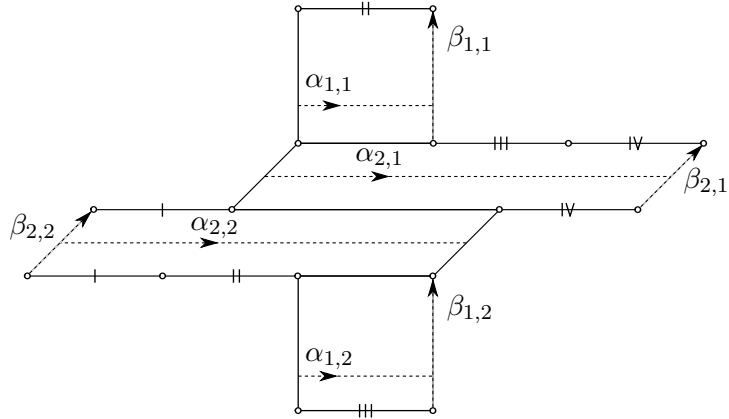
\caption{Model A cylinder decomposition of $(X,\omega)\in\Omega E_D(6)$.}\label{f.1}
\end{figure}

\begin{figure}[htb!]
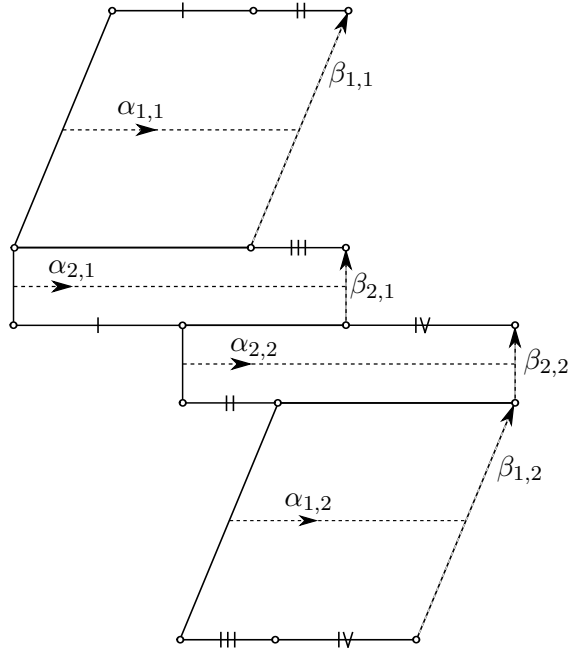
\caption{Model B cylinder decomposition of $(X,\omega)\in\Omega E_D(6)$.}\label{f.2}
\end{figure}

The receipt for the construction of $(X,\omega)\in\Omega E_D(6)$ from Models A and B is the following. Consider $w,h,t,e\in\mathbb{N}$ integral parameters such that $D=e^2+4wh$, $\gcd(w,h,t,e)=1$, $w,h>0$ and $0\leq t<\gcd(w,h)$.

If $w>2(e+2h)$, we can construct a translation surface $(X,\omega)\in\Omega E_D(6)$ from Model A by imposing that the homology cycles $\alpha_{i,j}, \beta_{i,j}$, $1\leq i,j\leq 2$ indicated in Figure~\ref{f.1} have holonomy vectors
\begin{eqnarray*}
\omega(\alpha_{1,1})=\omega(\alpha_{1,2})=(\lambda/2,0), \quad \omega(\beta_{1,1})=\omega(\beta_{1,2})=(0,\lambda/2) \\
\omega(\alpha_{2,1})=\omega(\alpha_{2,2})=(w/2,0), \quad \omega(\beta_{2,1})=\omega(\beta_{2,2})=(t/2,h/2)
\end{eqnarray*}
where $\lambda:=(e+\sqrt{D})/2$.

Similarly, if $h+e<w<2(e+2h)$, we can construct a translation surface $(X,\omega)\in\Omega E_D(6)$ from Model B by imposing that the homology cycles $\alpha_{i,j}, \beta_{i,j}$, $1\leq i,j\leq 2$ indicated in Figure~\ref{f.2} have holonomy vectors
\begin{eqnarray*}
\omega(\alpha_{1,1})=\omega(\alpha_{1,2})=(\lambda/2,0), \quad \omega(\beta_{1,1})=\omega(\beta_{1,2})=(0,\lambda/2) \\
\omega(\alpha_{2,1})=\omega(\alpha_{2,2})=(w/2,0), \quad \omega(\beta_{2,1})=\omega(\beta_{2,2})=(t/2,h/2)
\end{eqnarray*}
where $\lambda:=(e+\sqrt{D})/2$.

\begin{remark} Formally, the area of the Model A and B translation surfaces described above are not necessarily $1$. So, we need to \emph{rescale} them in order to get an element $(X,\omega)\in\Omega E_D(6)$. However, since this scaling issue is irrelevant in our subsequent discussion of Lyapunov exponents, we will ignore it in what follows. 
\end{remark}

For later use, we will say that the \emph{short}, resp. \emph{long}, cylinders in Models A, B are the two cylinders at the very top and very bottom, resp. in the middle, of the figures above whose waist curves have length $\lambda/2$, resp. $w/2$.

After this brief introduction to the Prym eigenforms, let us discuss the Lyapunov exponents of the Teichm\"uller curves inside the Weierstrass loci $\Omega E_D(2g-2)$.

\subsection{Lyapunov spectrum of Teichm\"uller curves in Weierstrass loci} For $(X,\omega)\in\Omega E_D(2g-2)$, we can use the holomorphic involution $\rho: X\to X$ to decompose the fiber $H_1(X,\mathbb{R})$ of the Hodge bundle at $(X,\omega)$ as
$H_1(X,\mathbb{R}) = H_1^-(X,\mathbb{R})\oplus H_1^+(X,\mathbb{R})$
where $H_1^-(X,\mathbb{R})$ consists of $\rho$-anti-invariant homology cycles and $H_1^+(X,\mathbb{R})$ consists of $\rho$-invariant homology cycles. By definition, $H_1^-(X,\mathbb{R})$ has dimension $4$ (as $\Omega^-(X)$ has complex dimension $2$) and, \emph{a fortiori}, $H_1^+(X,\mathbb{R})$ has dimension $2g-4$. Also, since $\omega\in\Omega^-(X)$, one has that $H_1^-(X,\mathbb{R}) = H_1^{st}(X,\mathbb{R})\oplus H^-$, where $H_1^{st}(X,\mathbb{R})$ is the tautological subspace and $H^-$ is the orthogonal of $H_1^{st}(X,\mathbb{R})$ inside $H_1^-(X,\mathbb{R})$ with respect to the intersection form on homology. Since $H_1^-(X,\mathbb{R})$ has dimension $4$ and $H_1^{st}(X,\mathbb{R})$ has dimension $2$, it follows that $H^-$ has also dimension $2$.

It is not hard to check that the decomposition
$$H_1(X,\mathbb{R})=H_1^{st}(X,\mathbb{R})\oplus H^-\oplus H_1^+(X,\mathbb{R})$$
is $\SL$-equivariant decomposition of the Hodge bundle over the Teichm\"uller curve $\SL\cdot (X,\omega)$ into symplectic subbundles. In particular, the Kontsevich-Zorich cocycle over $\SL\cdot (X,\omega)$ preserves each symplectic subbundle of this decomposition. As usual, the restriction of the Kontsevich-Zorich cocycle to the tautological subbundle $H_1^{st}(X,\mathbb{R})$ has Lyapunov exponents $\pm1$, and, therefore, the interesting Lyapunov exponents come from the restrictions of the Kontsevich-Zorich cocycle to $H^-$ and $H_1^+(X,\mathbb{R})$. 

In the genus $2$ case, it was shown by M. Bainbridge~\cite{B} and A. Eskin, M. Kontsevich and A. Zorich~\cite{EKZ} that the restriction of the Kontsevich-Zorich cocycle to $H^-$ has Lyapunov exponents $\pm1/3$. Since $H_1^+(X,\mathbb{R})$ is trivial in genus $2$, we completely understand the Lyapunov spectra of Teichm\"uller curves in the  Weierstrass loci $\Omega E_D(2)$.

In the genus $3$ case, after the works of M. M\"oller~\cite{Mo} and A. Eskin, M. Kontsevich and A. Zorich~\cite{EKZ}, it is known that the restriction of the Kontsevich-Zorich cocycle to $H^-$, resp. $H_1^+(X,\mathbb{R})$, has Lyapunov exponents $\pm1/5$, resp. $\pm2/5$. In particular, the Lyapunov spectra of Teichm\"uller curves in the Weierstrass loci $\Omega E_D(4)$ are also completely determined.

In the genus $4$ case, after the works of M. M\"oller~\cite{Mo} and A. Eskin, M. Kontsevich and A. Zorich~\cite{EKZ}, it is known that the restriction of the Kontsevich-Zorich cocycle to $H^-$ has Lyapunov exponents $\pm1/7$. On the other hand, since $H_1^+(X,\mathbb{R})$ is a $4$-dimensional symplectic subspace, the restriction of the Kontsevich-Zorich cocycle to $H_1^+(X,\mathbb{R})$ has Lyapunov spectrum $\lambda_1^+\geq\lambda_2^+\geq-\lambda_2^+\geq-\lambda_1^+$. Here, the work~\cite{EKZ} can be used to show that $\lambda_1^+ + \lambda_2^+ = 6/7$, and the work of G. Forni (see Corollary 6.8 of~\cite{Fo11}) can be applied to show that $\lambda_2^+>0$. However, the individual values of the exponents $\lambda_1^+$ and $\lambda_2^+$ are unknown. In particular, it is natural to ask about the simplicity of the Lyapunov exponents $\lambda_1^+$ and $\lambda_2^+$: in our case, this amounts to deciding whether $\lambda_1^+>\lambda_2^+$, that is, the simplicity of the ``top exponent'' $\lambda_1^+$ of the ``block'' $H_1^+(X,\mathbb{R})$. In fact, the simplicity of $\lambda_1^+$ is supported by numerical experiments with some examples (indicating that $\lambda_1^+\approx 0.58...$ and $\lambda_2^+\approx 0.26...$), and, in this direction, we have the following result:

\begin{theorem}\label{t.Prym} Consider the Teichm\"uller curves given by the $\SL$-orbits of Model B translation surfaces $(X,\omega)\in\Omega E_D(6)$ with parameters $t=0$, $e=1$, $h=(w-1)/3\in\mathbb{N}$. Then, for all but (possibly) finitely many values of the parameter $h\in\mathbb{N}$, one has $\lambda_1^+>\lambda_2^+>0$ (i.e., the Lyapunov spectrum of the Kontsevich-Zorich cocycle restricted to the block $H_1^+(X,\mathbb{Z})$ is simple).
\end{theorem}

\begin{remark} Note that the discriminant $D=e^2+4wh$ is $D=1+4(3h+1)h$ for the choice of parameters in Theorem~\ref{t.Prym}. In particular, by taking the reduction modulo $3$, we see that $D$ is not a square when $h\equiv 1$ (mod $3$). In other words, the statement of Theorem \ref{t.Prym} includes both non-arithmetic and arithmetic Teichm\"uller curves.
\end{remark}

\begin{remark} Roughly, Theorem~\ref{t.Prym} covers
    about a third of the known Teichm\"uller curves in $\Omega E_D(6)$.
It is likely that this simplicity result holds for all Teichm\"uller
curves in \emph{all} Weierstrass loci $\Omega E_D(6)$, but we will not
try to push our methods to get a more complete result 
because this is not the
main point of this note.
Indeed, the idea behind Theorem~\ref{t.Prym}
is to give an  example where Theorem~\ref{t.Lyapunov} can be applied,
while  the alternative approach via coding of the geodesic flow is
runs into serious difficulties since  for Teichm\"uller curves in
\emph{non-arithmetic} (i.e., $\sqrt{D}\notin\mathbb{N}$) Weierstrass
loci $\Omega E_D(6)$ the Veech group is unknown.
\end{remark}

The proof of Theorem~\ref{t.Prym} will occupy the rest of this section and it will consist of the following two steps. 

In the next subsection, we will start with a certain Model B translation surface $X=(M,\omega)\in\Omega E_D(6)$ and we will compute the action on homology of three parabolic elements of the Veech group $SL(M,\omega)$ related to two Model B cylinder decompositions and a Model A cylinder decomposition. The outcome of this calculation will be the following proposition:

\begin{proposition}\label{p.Prym-matrices}Let $X=(M,\omega)\in\Omega E_D(6)$ be a Model B translation surface with parameters $t=0$, $h, e\in\mathbb{N}$, $w=3h+e$. Then, the image of the monodromy representation $\rho:SL(M,\omega)\to H_1^+(X,\mathbb{R})$ contains the matrices 
\begin{equation*}
A=\left(\begin{array}{cccc}1 & 0 & w & 0 \\ 0 & 1 & 0 & (w-e)/3 \\ 0 & 0 & 1 & 0 \\ 0 & 0 & 0 & 1\end{array}\right), \quad
B=\left(\begin{array}{cccc}1 & 0 & 0 & 0 \\ 0 & 1 & 0 & 0 \\ 4w^2 & \frac{4w(w+2e)}{3} & 1 & 0 \\ \frac{4w(w+2e)}{3} & \frac{4w(w+2e)}{3} & 0 & 1\end{array}\right),
\end{equation*}
and
\begin{equation*}
C=\left(\begin{array}{cccc}1-\rho_{short} & -\rho_{short} & \rho_{short}+\rho_{long} & -\rho_{long} \\ 0 & 1 & -\rho_{long} &\rho_{long} \\ -\rho_{short} & -\rho_{short} & 1+\rho_{short} & 0 \\ -\rho_{short} & -\rho_{short} & \rho_{short} & 1\end{array}\right)
\end{equation*}
where 
$$\rho_{short}:=2(w+2e)^2(37w^3-69ew^2+45e^2w-13e^3)$$
and 
$$\rho_{long}:=(w-e)^2(37w^3+42e w^2-51e^2w+26e^3).$$
\end{proposition}

Then, in the final subsection, our next step is to prove (with the aid of some results in~\cite{MMY}) the following proposition:
\begin{proposition}\label{p.Prym-monoid} Consider the matrices $A$, $B$ and $C$ from Proposition~\ref{p.Prym-matrices} above in the case of parameters $t=0$, $e=1$, $h\in\mathbb{N}$ and $w=3h+1$. Then, any monoid $\mathcal{G}$ containing $A$, $B$ and $C$ is pinching and twisting for all $h\in\mathbb{N}$ large enough.
\end{proposition} 

At this stage, we will conclude the proof of Theorem~\ref{t.Prym} by simply combining Propositions~\ref{p.Prym-matrices} and~\ref{p.Prym-monoid} with item (c) of the ``coding-free'' simplicity criterion in Theorem~\ref{t.Lyapunov}.

\subsection{Action on homology of some Dehn twists of $(M,\omega)\in\Omega E_D(6)$}

Using the notations from Appendix D of E. Lanneau and D.-M. Nguyen's paper \cite{LN}, let $X=(M,\omega)\in\Omega E_D(6)$ be a Model B translation surface with $t=0$ (i.e., no ``twist'' between cylinders).

We consider the basis of homology $\{\alpha_{i,j}, \beta_{i,j}: 1\leq i,j\leq 2\}$ of $X=(M,\omega)$ from Figure~\ref{f.2} above, so that
$$\alpha_i:=\alpha_{i,1}+\alpha_{i,2}, \quad \beta_i:=\beta_{i,1}+\beta_{i,2}, \quad i=1,2$$
is a basis of $H_1^{-}(X,\mathbb{Z})$, and
$$\widetilde{\alpha}_i:=\alpha_{i,1}-\alpha_{i,2}, \quad \widetilde{\beta}_i:=\beta_{i,1}-\beta_{i,2}, \quad i=1,2$$
is a basis of $H_1^+(X,\mathbb{Z})$.

As we already explained, for our purposes (of showing Theorem~\ref{t.Prym}), we will focus mostly on the action of $\textrm{Aff}(X,\omega)$ on $H_1^+(X,\mathbb{Z})$.

\begin{remark} It is not hard to see that translation surfaces in \emph{minimal strata} have a trivial automorphism groups. In particular, from the discussion of \S\ref{s.Lyapunov}, we have that the Veech group of any $(X,\omega)\in\Omega E_D(6)$ injects into its group of affine diffeomorphisms. In other words, in the case of $(X,\omega)\in\Omega E_D(6)$, we can study the Kontsevich-Zorich cocycle by analyzing how the affine diffeomorphisms naturally associated to elements of $SL(X,\omega)$ act on $H_1(X,\mathbb{R})$.
\end{remark}

Recall that, in Model B (with $t=0$), the holonomy vectors of $\alpha_{i,j}, \beta_{i,j}$ are
\begin{eqnarray*}
\omega(\alpha_{1,1})=\omega(\alpha_{1,2})=(\lambda/2,0), \omega(\beta_{1,1})=\omega(\beta_{1,2})=(0,\lambda/2) \\
\omega(\alpha_{2,1})=\omega(\alpha_{2,2})=(w/2,0), \omega(\beta_{2,1})=\omega(\beta_{2,2})=(0,h/2)
\end{eqnarray*}

Firstly, we look in the horizontal direction and we consider the parabolic element $A:=\left(\begin{array}{cc}1 & w \\ 0 & 1\end{array}\right)$. The moduli of the horizontal cylinders of $X$ are $1$ and $w/h$. It follows that $A$ belongs to $\textrm{SL}(X,\omega)$ and it acts on homology as
$$A(\alpha_{i,j})=\alpha_{i,j}$$
and
$$A(\beta_{1,j})=w\alpha_{1,j} + \beta_{1,j}, \quad A(\beta_{2,j})=h\alpha_{2,j} + \beta_{2,j}$$
In particular, the matrix of $A$ on the basis $\{\widetilde{\alpha}_1, \widetilde{\alpha}_2, \widetilde{\beta}_1, \widetilde{\beta}_2\}$ of $H_1^+(X,\mathbb{Z})$ is
$$A=\left(\begin{array}{cccc}1 & 0 & w & 0 \\ 0 & 1 & 0 & h \\ 0 & 0 & 1 & 0 \\ 0 & 0 & 0 & 1\end{array}\right) = \left(\begin{array}{cc}\textrm{Id}_{2\times 2} & \widetilde{A} \\ 0 & \textrm{Id}_{2\times 2}\end{array}\right)$$
where $\textrm{Id}_{2\times 2}$ is the $2\times 2$ identity matrix and $\widetilde{A}=\left(\begin{array}{cc}w & 0 \\ 0 & h\end{array}\right)$.

Secondly, we look in the vertical direction and we consider the parabolic element $B:=\left(\begin{array}{cc}1 & 0 \\ \tau & 1\end{array}\right)$ where
$$\tau := 4(D-(w-e)^2)\frac{\lambda/2+h/2}{\lambda-w/2} = ((2w-e)^2-D)\frac{\lambda/2+h}{w/2-\lambda/2}$$ Here, we used the relations $2\lambda=e+\sqrt{D}$ and
$D=e^2+4wh$. Note that the moduli of the vertical cylinders are $\frac{\lambda/2+h/2}{\lambda-w/2}$ and $\frac{\lambda/2+h}{w/2-\lambda/2}$, so that $B$ belongs to $\textrm{SL}(X,\omega)$ and it acts on homology as
$$B(\beta_{i,j})=\beta_{i,j}$$
and
\begin{eqnarray*}
&B(\alpha_{1,j})= \alpha_{1,j}+ 4(D-(w-e)^2)\alpha_{1,j}^{vert} + ((2w-e)^2-D)\alpha_{2,j}^{vert}, \\
&B(\alpha_{2,j})= \alpha_{2,j}+ 4(D-(w-e)^2)\alpha_{1,j}^{vert}+ ((2w-e)^2-D)(\alpha_{2,1}^{vert} + \alpha_{2,2}^{vert})
\end{eqnarray*}
where $\alpha_{1,j}^{vert}:=\beta_{1,j}+\beta_{2,j}$ are the homology classes of short vertical cylinders and $\alpha_{2,j}^{vert}:=\beta_{1,j}+\beta_{2,1}+\beta_{2,2}$ are the homology classes of long vertical cylinders. In particular, the matrix of $B$ on the basis $\{\widetilde{\alpha}_1, \widetilde{\alpha}_2, \widetilde{\beta}_1, \widetilde{\beta}_2\}$ of $H_1^+(X,\mathbb{Z})$ is
\begin{eqnarray*}
B&=&\left(\begin{array}{cccc}1 & 0 & 0 & 0 \\ 0 & 1 & 0 & 0 \\ 4(D-(w-e)^2) + ((2w-e)^2-D) & 4(D-(w-e)^2) & 1 & 0 \\ 4(D-(w-e)^2) & 4(D-(w-e)^2) & 0 & 1\end{array}\right) \\
&=& \left(\begin{array}{cc}\textrm{Id}_{2\times 2} & 0 \\ \widetilde{B} & \textrm{Id}_{2\times 2}\end{array}\right)
\end{eqnarray*}
where $\textrm{Id}_{2\times 2}$ is the $2\times 2$ identity matrix and
\begin{eqnarray*}
\widetilde{B}&=&\left(\begin{array}{cc}4(D-(w-e)^2) + ((2w-e)^2-D) & 4(D-(w-e)^2) \\ 4(D-(w-e)^2) & 4(D-(w-e)^2)\end{array}\right) \\ &=&
\left(\begin{array}{cc}4w(3h+e) & 4w(4h-w+2e) \\ 4w(4h-w+2e) & 4w(4h-w+2e)\end{array}\right).
\end{eqnarray*}
Here, we used the relation $D=e^2+4wh$.

Finally, we consider the ``diagonal'' direction of slope $\theta:=(\lambda+h)/\lambda$. In general, the cylinder decomposition in this direction is given by a Model A.

For sake of simplicity, we will take $3h=w-e$, so that $\theta\cdot (w/2-\lambda/2)=h/2$ and thus the cylinder decomposition is the one presented in Figure~\ref{f.3} below.

\begin{figure}[htb!]
\includegraphics{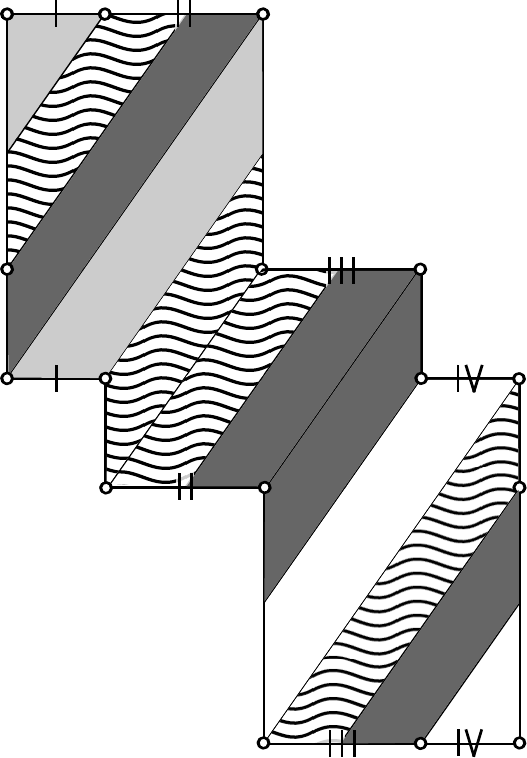}
\caption{Cylinder decomposition in direction $\theta$.}\label{f.3}
\end{figure}

The homology classes of short cylinders in ``diagonal'' direction $\theta$ are
$$\alpha_{1,j}^{diag}:=\alpha_{1,j}+\beta_{1,j}+\beta_{2,j},$$
and the homology classes of long cylinders in ``diagonal'' direction $\theta$ are
$$\alpha_{2,j}^{diag}:=\alpha_{1,1}^{diag}+\alpha_{1,2}^{diag}+\alpha_{2,j}-\alpha_{1,j}+\beta_{2,1}+\beta_{2,2}$$

Also, a short cylinder in direction $\theta$ has modulus
$$\mu_{short}^{diag}:=\frac{\left(\frac{\lambda}{2}\right)^2 + \left(\frac{\lambda}{2}+\frac{h}{2}\right)^2}{\left(\lambda-\frac{w}{2}\right)\left(\frac{\lambda}{2}+\frac{h}{2}\right)}$$
and a long cylinder in direction $\theta$ has modulus
$$\mu_{long}^{diag}:=2\frac{\left(\frac{w}{2}+\frac{\lambda}{2}\right)^2 + \left(\lambda+2h\right)^2}{\left(\frac{w}{2}-\frac{\lambda}{2}\right)\left(\lambda+2h\right)}$$

A direct computation using the relations $3h=w-e$, $D=e^2+4wh = e^2+4w(w-e)/3$ and $2\lambda=e+\sqrt{D}$ reveals that
$$\frac{\mu_{long}^{diag}}{\mu_{short}^{diag}}=\frac{2(w+2e)^2(37w^3-69ew^2+45e^2w-13e^3)}{(w-e)^2(37w^3+42e w^2-51e^2w+26e^3)}$$
This motivates the consideration of the action on homology of the element $C$ of $\textrm{SL}(X,\omega)$ given by the Dehn twist in direction $\theta$ by
$$\rho:=\rho_{long}\cdot\mu_{long}^{diag} = \rho_{short}\cdot\mu_{short}^{diag}$$
where
$$\rho_{short}:=2(w+2e)^2(37w^3-69ew^2+45e^2w-13e^3)$$
and 
$$\rho_{long}:=(w-e)^2(37w^3+42e w^2-51e^2w+26e^3).$$

By inspecting how the cylinders in direction $\theta$ intersect the cycles $\alpha_{i,j}, \beta_{i,j}$, we get that
$$C(\alpha_{1,j})=\alpha_{1,j} - \rho_{short}\alpha_{1,j}^{diag} - \rho_{long}(\alpha_{2,1}^{diag}+\alpha_{2,2}^{diag}),$$ 
$$C(\beta_{1,1})=\beta_{1,1} + \rho_{short}\alpha_{1,1}^{diag}+\rho_{long}\alpha_{2,2}^{diag},$$ 
$$C(\beta_{1,2})=\beta_{1,2} + \rho_{short}\alpha_{1,2}^{diag}+\rho_{long}\alpha_{2,1}^{diag},$$
$$C(\alpha_{2,j})=\alpha_{2,j} - \rho_{short}\alpha_{1,j}^{diag} - 2\rho_{long}(\alpha_{2,1}^{diag}+\alpha_{2,2}^{diag}),$$ 
$$C(\beta_{2,j})=\beta_{2,j} + \rho_{long}\alpha_{2,j}^{diag}$$
Since $\alpha_{1,1}^{diag}-\alpha_{1,2}^{diag}=\widetilde{\alpha}_1 + \widetilde{\beta}_1 + \widetilde{\beta}_2$ and $\alpha_{2,1}^{diag}-\alpha_{2,2}^{diag}=\widetilde{\alpha}_2-\widetilde{\alpha}_1$, we obtain that the matrix of $C$ on the basis $\{\widetilde{\alpha}_1, \widetilde{\alpha}_2, \widetilde{\beta}_1, \widetilde{\beta}_2\}$ of $H_1^+(X,\mathbb{Z})$ is
$$C=\left(\begin{array}{cccc}1-\rho_{short} & -\rho_{short} & \rho_{short}+\rho_{long} & -\rho_{long} \\ 0 & 1 & -\rho_{long} &\rho_{long} \\ -\rho_{short} & -\rho_{short} & 1+\rho_{short} & 0 \\ -\rho_{short} & -\rho_{short} & \rho_{short} & 1\end{array}\right)$$

In resume, we started with a Model B with parameters $h=(w-e)/3$, $t=0$ and we got the matrices
\begin{equation}\label{e.A}
A=\left(\begin{array}{cccc}1 & 0 & w & 0 \\ 0 & 1 & 0 & (w-e)/3 \\ 0 & 0 & 1 & 0 \\ 0 & 0 & 0 & 1\end{array}\right) =: \left(\begin{array}{cc}\textrm{Id}_{2\times 2} & \widetilde{A} \\ 0 & \textrm{Id}_{2\times 2}\end{array}\right),
\end{equation}
\begin{equation}\label{e.B}
B=\left(\begin{array}{cccc}1 & 0 & 0 & 0 \\ 0 & 1 & 0 & 0 \\ 4w^2 & \frac{4w(w+2e)}{3} & 1 & 0 \\ \frac{4w(w+2e)}{3} & \frac{4w(w+2e)}{3} & 0 & 1\end{array}\right)=: \left(\begin{array}{cc}\textrm{Id}_{2\times 2} & 0 \\ \widetilde{B} & \textrm{Id}_{2\times 2}\end{array}\right),
\end{equation}
and
\begin{equation}\label{e.C}
C=\left(\begin{array}{cccc}1-\rho_{short} & -\rho_{short} & \rho_{short}+\rho_{long} & -\rho_{long} \\ 0 & 1 & -\rho_{long} &\rho_{long} \\ -\rho_{short} & -\rho_{short} & 1+\rho_{short} & 0 \\ -\rho_{short} & -\rho_{short} & \rho_{short} & 1\end{array}\right)
\end{equation}
corresponding to the action of three elements of $\textrm{SL}(X,\omega)$ (coming from certain Dehn twists in three periodic directions) on $H_1^+(X,\mathbb{Z})$. In other words, we proved Proposition~\ref{p.Prym-matrices}.

In the next subsection, we will use these matrices to complete the proof of Theorem~\ref{t.Prym}.

\subsection{End of the proof of Theorem~\ref{t.Prym}}

By Theorem~\ref{t.Lyapunov}, we can complete the proof of Theorem~\ref{t.Prym} by considering Model B with parameters $t=0$, $e=1$, $h=(w-1)/3\in\mathbb{N}$ and by showing the pinching and twisting properties for any monoid $\mathcal{G}$ containing the matrices $A$, $B$ and $C$ given by \eqref{e.A}, \eqref{e.B} and \eqref{e.C} for adequate values of the ``free'' parameter $w\in\mathbb{N}$ (or equivalently $h=(w-1)/3\in\mathbb{N}$). In this direction, we'll need the following ``Galois-theoretical'' criterion for the pinching and twisting properties from the article~\cite{MMY}:

\begin{theorem}\label{t.MMY}Let $\mathcal{G}$ be a monoid containing two matrices $M, N\in Sp(4,\mathbb{Z})$. Denote by $P(x)=x^4 + a(P) x^3 + b(P) x^2 + a(P) x +1$ and $Q(x)=x^4 + a(Q) x^3 + b(Q) x^2 + a(Q) x +1$ the characteristic polynomials of $M$ and $N$. Suppose that the discriminants 
\begin{itemize} 
\item $\Delta_1(P):= a(P)^2 - 4 (b(P)-2)$, $\Delta_2(P):=(b(P)+2)^2-  4 a(P)^2$, $\Delta_3(P):=\Delta_1(P)\cdot \Delta_2(P)$,
\item $\Delta_1(Q):= a(Q)^2 - 4 (b(Q)-2)$, $\Delta_2(Q):=(b(Q)+2)^2-  4 a(Q)^2$, $\Delta_3(Q):=\Delta_1(Q)\cdot \Delta_2(Q)$, and 
\item $\Delta_i(P)\cdot\Delta_j(Q)$, $1\leq i,j\leq 3$ 
\end{itemize}
are positive integers that are not squares. Then, the matrices $M$ and $N$ are pinching and some product of powers of $M$ and $N$ is twisting with respect to $M$, and, \emph{a fortiori}, the monoid $\mathcal{G}$ has the pinching and twisting properties.
\end{theorem}

In our context, we consider the matrix $M=A\cdot B$ where $A$ and $B$ are the matrices presented in \eqref{e.A} and \eqref{e.B}. It has the form
$$M=\left(\begin{array}{cc}\textrm{Id}_{2\times 2}+\widetilde{A}\widetilde{B} & \widetilde{A} \\ \widetilde{B} & \textrm{Id}_{2\times 2}\end{array}\right)$$
and, therefore, its characteristic polynomial is
\begin{eqnarray*}
& &P(x):= x^4 + a(P) x^3 + b(P) x^2 + a(P) x +1 = \\
& &x^4-(\textrm{tr}(\widetilde{A}\widetilde{B})+4) x^3 + (\textrm{det}(\widetilde{A}\widetilde{B}) + 2\textrm{tr}(\widetilde{A}\widetilde{B})+6)x^2 - (\textrm{tr}(\widetilde{A}\widetilde{B})+4)x+1.
\end{eqnarray*}
In our case,
$$\widetilde{A}\widetilde{B}= \left(\begin{array}{cc} 4w^3 & 4 w^2(w+2)/3 \\ 4(w-1)w(w+2)/9 & 4(w-1)w(w+2)/9 \end{array}\right)$$
so that 
$$\textrm{tr}(\widetilde{A}\widetilde{B})=4w(10w^2+w-2)/9$$ 
and 
$$\textrm{det}(\widetilde{A}\widetilde{B})=32w^3(w^3-3w+2)/27.$$

We compute the following discriminants
\begin{eqnarray*}
\Delta_1(P) &:=& a(P)^2-4(b(P)-2) := \textrm{tr}(\widetilde{A}\widetilde{B})^2 - 4\,\textrm{det}(\widetilde{A}\widetilde{B}) \\
&=& \frac{16}{81}w^2(76w^4+20w^3+33w^2-52w+4) \\ 
&=& 16(3h+1)^2 (76 h^4 + 108 h^3 + 61 h^2 + 14 h +1)
\end{eqnarray*}
\begin{eqnarray*}
& &\Delta_2(P) := (b(P)+2)^2 - 4a(P)^2 = \\ && \frac{512}{729}(w-1)^2 w^3 (2w^7+4w^6-6w^5+22w^4+71w^3+15w+54)
 \end{eqnarray*}
and
\begin{equation*}
\Delta_3(P) := \Delta_1(P)\cdot \Delta_2(P)
\end{equation*}

Next, we consider the matrix $N=B\cdot C$ where $B$ and $C$ the matrices presented in \eqref{e.B} and \eqref{e.C}. Its characteristic polynomial is 
\begin{eqnarray*}
& &Q(x):= x^4 + a(Q) x^3 + b(Q) x^2 + a(Q) x +1
\end{eqnarray*}
where 
\begin{eqnarray*}
a(Q)= - \frac{1184 w^7}{3} - 448 w^6 + 
 840 w^5 + \frac{472 w^4}{3} - 296 w^3  + 72 w^2 + \frac{208 w}{3}  -4   
\end{eqnarray*}
and
\begin{eqnarray*}
b(Q) &=& \frac{87616 w^{14}}{9} + \frac{66304 w^{13}}{3} - 72704 w^{12} - \frac{994048 w^{11}}{9} + 295808 w^{10} \\ &+& 106752 w^9 - \frac{1864192 w^8}{3} + 
\frac{1076800 w^7}{3} 
 + 347072 w^6 - \frac{5648656 w^5}{9} \\ &+& 405360 w^4 - 132528 w^3 + \frac{171760 w^2}{9} - \frac{416 w}{3} + 6
\end{eqnarray*} 
Again, we compute the following discriminants:
\begin{eqnarray*}
\Delta_1(Q)&:=& a(Q)^2-4(b(Q)-2) \\ &=& \frac{64}{3} ((w-1) w (w+2) (37w^2 - 32w + 13))^2\cdot \\ &\cdot & (2w^2 +2w -1) (2 w^2 + 2 w+5),  
\end{eqnarray*}
\begin{eqnarray*}
\Delta_2(Q)&:=& (b(P)+2)^2 - 4a(P)^2 \\ &=& \frac{1024}{81} (w-1)^4 w^2 (w+2)^4 (13 - 32 w + 37 w^2)^2 \cdot \\ &\cdot & \left(5476 w^{14} + 12432 w^{13} - 
   40896 w^{12} - 62128 w^{11} + 166392 w^{10} + \right. \\ & & \left. 60048 w^9 - 349536 w^8 + 202344 w^7 + 195732 w^6 - 353986 w^5 +\right. \\ & & \left. 227838 w^4 - 74214 w^3 + 
   10654 w^2 - 156 w + 9\right) 
\end{eqnarray*}
and 
$$\Delta_3(Q):=\Delta_1(Q)\cdot\Delta_2(Q)$$

Given a polynomial $R$ with rational coefficients with factorization $R(x)=\prod\limits_{k=1}^m R_k(x)^{a_k}$ over $\mathbb{Q}[x]$, denote by $R^{red}(x):=\prod\limits_{k=1}^m R_k(x)^{a_k (\textrm{mod }2)}$ its \emph{square-free reduction}. Note that the values of the parameter $x\in\mathbb{Z}$, resp. $\mathbb{Q}$, such that $R(x)$ is a square correspond to integral, resp. rational, points of the curve $z^2=R^{red}(x)$. In particular, by Siegel's theorem, if $R^{red}(x)$ has degree $\geq 3$, then $R(x)$ is not a square for all but finitely many values of $x\in\mathbb{Z}$, and, by Falting's theorem, if $R^{red}(x)$ has degree $\geq 5$, then $R(x)$ is not a square for all but finitely many values of $x\in\mathbb{Q}$. 

In our setting, we have that 
\begin{itemize}
\item $\Delta_1(P)^{red}(h) = 76 h^4 + 108 h^3 + 61 h^2 + 14 h +1$
\item $\Delta_1(P)^{red}(w) = 76w^4+20w^3+33w^2-52w+4$
\item $\Delta_2(P)^{red}(w) = w(2w^7+4w^6-6w^5+22w^4+71w^3+15w+54)$
\item $\Delta_1(Q)^{red}(w) = 3(2w^2 +2w -1) (2 w^2 + 2 w+5)$
\item $\Delta_2(Q)^{red}(w) = 5476 w^{14} + 12432 w^{13} - 
   40896 w^{12} - 62128 w^{11} \\ + 166392 w^{10} + 60048 w^9 - 349536 w^8 + 202344 w^7 + 195732 w^6 \\ - 353986 w^5 +   227838 w^4 - 74214 w^3 + 
   10654 w^2 - 156 w + 9$
\end{itemize} 
where all polynomials written in the five items above are irreducible over $\mathbb{Q}[x]$. In particular, since $\Delta_1(P)^{red}(h)$, $\Delta_1(P)^{red}(w)$ and $\Delta_1(Q)^{red}(w)$ have degree $4$, $\Delta_2(P)^{red}(w)$ has degree $8$, $\Delta_2(Q)^{red}(w)$ has degree $14$, and they don't have common factors, the discussion of the previous paragraph (based on Siegel's theorem and Faltings' theorem) apply to ensure that the discriminants 
\begin{itemize} 
\item $\Delta_1(P)$, $\Delta_2(P)$, $\Delta_3(P)=\Delta_1(P)\cdot \Delta_2(P)$,
\item $\Delta_1(Q)$, $\Delta_2(Q)$, $\Delta_3(Q):=\Delta_1(Q)\cdot \Delta_2(Q)$, and 
\item $\Delta_i(P)\cdot\Delta_j(Q)$, $1\leq i,j\leq 3$ 
\end{itemize}
are not squares for all but finitely many values of $w\in\mathbb{N}$ (or $h\in\mathbb{N}$). 

Moreover, the leading coefficients of $\Delta_i(P)^{red}$ and $\Delta_j(Q)^{red}$ are positive, so that these discriminants are positive for all but finitely many values of $w\in\mathbb{N}$ (or $h\in\mathbb{N}$). 

Therefore, since any monoid $\mathcal{G}$ containing the matrices $A$, $B$ and $C$ must contain the matrices $M:=A\cdot B$ and $N:=B\cdot C$ (whose characteristic polynomials are $P$ and $Q$ above), we can use Theorem~\ref{t.MMY} to conclude that any monoid $\mathcal{G}$ containing the matrices $A$, $B$ and $C$ is pinching and twisting. This proves Proposition~\ref{p.Prym-monoid}. 

Finally, as we already mentioned, by Theorem~\ref{t.Lyapunov}, the proof of Theorem~\ref{t.Prym} is also complete.

\section{Lyapunov exponents of variations of Hodge structures of higher weight}

In some recent talks, M. Kontsevich \cite{Kont13} discussed the possibility of extending the formula in \cite{EKZ} for sums of (non-negative) Lyapunov exponents of Teichm\"uller curves to more general contexts including variations of Hodge structures of higher weights. 

Abstractly, Kontsevich considers the following scenario. Let $C$ be a hyperbolic Riemann surface of finite area. Denote by $(\mathcal{E}, \nabla)$ a vector bundle $\mathcal{E}$ over $C$ with a flat connection $\nabla$. Observe that the data of $(\mathcal{E},\nabla)$ gives a linear cocycle over the geodesic flow on the hyperbolic Riemann surface $C$ and also a monodromy representation $\rho:\pi_1(C,c_0)\to GL(N,\mathbb{C})$ where $c_0\in C$ and $N$ is the rank of $\mathcal{E}$. Assuming that the fibers $\mathcal{E}_x$ of $\mathcal{E}$ are equipped with a measurable family of norms $\|.\|_x$ that are bounded near the cusps of $C$, one can check that the linear cocycle over the geodesic flow on $C$ induced by $(\mathcal{E},\nabla)$ satisfies the $L^1$ $\log$-integrability condition in Oseledets theorem whenever the monodromy representation $\rho:\pi_1(C,c_0)\to GL(N,\mathbb{C})$ is \emph{quasi-unipotent} near the cusps of $C$ (i.e., the spectra of the matrices obtained as images under $\rho$ of small loops around the cusps of $C$ are contained in the unit circle in the complex plane). In this context, we can apply Oseledets theorem to get Lyapunov exponents $\lambda_1\geq\dots\geq\lambda_N$, $N=\textrm{rank}(\mathcal{E})$, associated to the linear cocycle induced by $(\mathcal{E},\nabla)$ over $C$, and it is a natural question to try to compute these Lyapunov exponents using some geometrical information on $(\mathcal{E},\nabla)$.

One of the main examples of the situation described in the previous paragraph are \emph{variations of Hodge structures} associated to one-parameter deformations of compact K\"ahler manifolds. Here, the presentation of these examples will follow Voisin's book \cite{Voisin} (that we refer for basic definitions and more details). Let us consider a family $X_{c}$, $c\in C$, of (mutually diffeomorphic) compact K\"ahler manifolds parametrized by a hyperbolic Riemann surface $C$. The cohomology groups $H^k(X_c,\mathbb{C})$ form the fibers of a vector bundle $\mathcal{H}^k$ equipped with the (flat) Gauss-Mannin connection. Furthermore, the fibers $H^k(X_c,\mathbb{C})$ come with an integer lattice $H^k(X_c,\mathbb{C})=H^k(X_c,\mathbb{Z})\otimes\mathbb{C}$, and they have a \emph{Hodge decomposition} 
$$H^k(X_c,\mathbb{C})=\bigoplus\limits_{p+q=k}H^{p,q}(X_c)$$
where $H^{p,q}$ is the space of cohomology classes of type $(p,q)$ and a \emph{Hodge filtration} 
$$F^pH^k(X_c,\mathbb{C})=\bigoplus\limits_{r\geq p} H^{r,k-r}(X_c)$$
In the literature, this example is usually given when introducing the notion of general \emph{variations} of (integral) \emph{Hodge structures of weight $k$}. 

\begin{example}The Hodge bundle $H^1_g$ equipped with the Gauss-Manin connection over a Teichm\"uller curve is a variation of Hodge structures of weight one.
\end{example}

It is possible to show that the monodromy representations corresponding to variations of Hodge structures described above are quasi-unipotent near the cusps (see, e.g., Theorem 15.15 of Voisin's book \cite{Voisin}), so that we can use Oseledets theorem to talk about Lyapunov exponents associated to these monodromy representations.

As we already mentioned, the article \cite{EKZ} contains formulas for the sums of non-negative Lyapunov exponents of the Kontsevich-Zorich cocycle. Very roughly speaking, Kontsevich (and Forni) showed a \emph{formula} relating the sums of non-negative Lyapunov exponents of the Kontsevich-Zorich cocycle to the integral of the first Chern class of the ``middle part'' $F^1H^1_g:=H^{1,0}$ of the Hodge filtration.

Of course, as M\"oller suggested to Kontsevich, it is natural to try to generalize this formula for the sum of non-negative Lyapunov exponents associated to \emph{variations of Hodge structures} of \emph{higher weights}.

In order to test his ideas, Kontsevich studies certain prototypical algebro-geometrical examples of \emph{Calabi-Yau $3$-folds} (3CY for short)\footnote{Recall that a Calabi-Yau $n$-fold is a compact K\"ahler manifold of complex dimension $n$ with vanishing Ricci curvature.}. More concretely, there are \emph{several} families of 3CYs, and, among those, one finds 14 families of 3CY whose moduli spaces are isomorphic to $\overline{\mathbb{C}}-\{0,1,\infty\}$ (see, e.g., \cite{DM}). For each of these families $X^{(k)}_{c}$, $k=1,\dots,14$, $c\in C\simeq\overline{\mathbb{C}}-\{0,1,\infty\}$, we have the vector bundle $H^3(X^{(k)})$ whose fibers are the third cohomology groups $H^3(X^{(k)}_c,\mathbb{C})$ and  the Gauss-Manin connection over the hyperbolic Riemann surface $C:=\overline{\mathbb{C}}-\{0,1,\infty\}$. Thus, it makes sense to talk about the Lyapunov exponents in this context. 

Interestingly enough, Kontsevich found that the natural generalization of the formula for the sum of non-negative Lyapunov exponents in terms of the first Chern class of the ``middle part'' $F^2H^3:=H^{3,0}\oplus H^{2,1}$ of the Hodge filtration works exactly for $7$ of the $14$ families $X_c^{(k)}$, $k=1,\dots, 14$, and, as a matter of fact, the formula works precisely in the cases when the image of the corresponding monodromy representation $\rho$ is a ``thin'' group (in Sarnak's terminology). Also, Kontsevich observed that, in the remaining $7$ cases where the formula does not work, the sum of non-negative Lyapunov exponents are \emph{strictly} larger than the quantity provided by the first Chern class of the ``middle part'' $F^2H^3$ of the Hodge filtration (i.e., his ``formula'' becomes a strict lower bound in the $7$ ``bad'' cases). 

Among these 14 families $X_c^{(k)}$, $k=1,\dots,14$, of 3CY's parametrized by $c\in C\simeq\overline{\mathbb{C}}-\{0,1,\infty\}$, one has the so-called \emph{mirror quintic} 3CY's. In the literature, mirror quintic 3CY's were introduced in \cite{COGP} in their study of \emph{mirror symmetry}. 

For our purposes, it suffices to know that the Hodge numbers of a mirror quintic 3CY $X_c^{mq}$, $c\in C\simeq\overline{\mathbb{C}}-\{0,1,\infty\}$, are $h^{0,3}=h^{1,2}=h^{2,1}=h^{3,0}=1$ (where $h^{p,q}$ denotes the complex dimension of $H^{p,q}(X_c^{mq})$, so that the third cohomology group $H^3(X_c^{mq})$ of a mirror quintic 3CY $X_c^{mq}$ is four-dimensional. Therefore, since the natural intersection form on the integral lattice $H^3(X_c^{mq},\mathbb{Z})$ (induced by the cup product) is symplectic, we obtain a monodromy representation 
$$\rho:\pi_1(C)\to Sp(4,\mathbb{Z})$$
associated to the variations of Hodge structures of weight three of mirror quintic 3CY's. 

As we discussed above, one can use $\rho:\pi_1(C)\to Sp(4,\mathbb{Z})$ to define a cocycle over the geodesic flow on the hyperbolic Riemann surface $C\simeq\overline{\mathbb{C}}-\{0,1,\infty\}\simeq\mathbb{H}/\Gamma_0(2)$ (where $\Gamma_0(2)$ is the subgroup of $SL(2,\mathbb{Z})$ consisting of matrices whose lower-left entry is zero modulo two). By definition, the Lyapunov spectrum of this cocycle has the form $\lambda_1\geq\lambda_2\geq-\lambda_2\geq-\lambda_1$, and Kontsevich's formula for the sum $\lambda_1+\lambda_2$ in terms of the geometry of $\rho$ (first Chern class of the middle part of the Hodge filtration) \emph{works}. From his arguments, it is possible to show that $\lambda_1+\lambda_2>0$, so that $\lambda_1>0$. On the other hand, it is not clear how to deduce qualitative information on $\lambda_2$ (e.g., $\lambda_2>0$ and/or $\lambda_1>\lambda_2$) from the formula of Kontsevich.

In this section, we will show that Theorems \ref{t.Lyapunov} and \ref{t.MMY} allows us to deduce simplicity of the Lyapunov spectrum for a variations of Hodge structures of weight three of mirror quintic 3CY's. 

The main result of this section is:

\begin{theorem}\label{t.mirror-quintic} The Lyapunov spectrum $\lambda_1\geq\lambda_2\geq-\lambda_2\geq-\lambda_1$ of the monodromy representation of mirror quintic 3CY's is simple (i.e., $\lambda_1>\lambda_2>0$).
\end{theorem}

\begin{proof} For the proof of this result, we will need the following fact. For mirror quintic 3CY's, the corresponding monodromy representation $\rho:\pi_1(C)\to Sp(4,\mathbb{Z})$ was computed in several places of the literature (cf. \cite{CYY}, \cite{DM}, \cite{Movasati}, and \cite{vEvS}) and, for example, in \cite{Movasati}, it is shown that the following matrices 
$$M_0=\left(\begin{array}{cccc}1&1&0&0 \\ 0&1&0&0 \\ 5&5&1&0 \\ 0&-5&-1&1\end{array}\right)$$
and 
$$M_1=\left(\begin{array}{cccc}1&0&0&0 \\ 0&1&0&1 \\ 0&0&1&0 \\ 0&0&0&1\end{array}\right)$$
correspond to the image under $\rho$ of small loops in $C$ around $0$ and $1$. In particular, the image of the representation $\rho$ associated to mirror quintic 3CY's is the group generated by $M_0$ and $M_1$. 

We affirm that there are two matrices $A$ and $B$ in the group generated by $M_0$ and $M_1$ fitting the hypothesis of Theorem \ref{t.MMY}. In fact, this is not hard to show: for example, the matrices $A:=M_0^3\cdot M_1\in\mathcal{G}$ and $B:=M_0^4\cdot M_1\in\mathcal{G}$ have characteristic polynomials  
$$P(x) = x^4+31 x^3 + 71 x^2 + 31 x + 1$$
and 
$$Q(x) = x^4 + 66 x^3 + 186 x^2 + 66 x + 1$$
(resp.). Therefore, the corresponding discriminants are 
$$\Delta_1(P) = 685 = 5\times 137, \quad \Delta_2(P) = 1485 = 3^3\times 5\times 11$$
and 
$$\Delta_1(Q) = 3620 = 2^2\times 5\times 181, \quad \Delta_2(P) = 17920 = 2^9\times 5\times 7$$
so that none of the positive numbers $\Delta_i(P), \Delta_j(Q)$, and $\Delta_i(P)\Delta_j(Q)$ (for $1\leq i, j\leq 3$) are squares. 

Once we know that the image of $\rho$ is pinching and twisting, we want to apply to Theorem \ref{t.Lyapunov} to conclude the simplicity of the Lyapunov exponents. However, this is not completely straightforward because Theorem \ref{t.Lyapunov} concerns variations of Hodge structures of weight one, while our current setting concerns variations of Hodge structures of weight three. 

Fortunately, it is not difficult to adapt the proof of Theorem \ref{t.Lyapunov} to this case. Indeed, the fact that random products of loops around the cusps $0$, $1$ and $\infty$ of $C\simeq\overline{\mathbb{C}}-\{0,1,\infty\}\simeq\mathbb{H}/\Gamma_0(2)$ track geodesic rays in $\mathbb{H}/\Gamma_0(2)$ is still true (by Oseledets theorem), and, hence, we have only to justify the validity of an analog of ``Forni estimate'' saying that the norms of the monodromy matrices are controlled by the distances in the hyperbolic plane between random walks and geodesic rays (cf. \eqref{e.Forni-sublinear-error}). As it turns out, this is a consequence of the following argument\footnote{Actually, the same argument can be used to obtain an algebro-geometrical proof of Forni's estimate.}. We have a \emph{period map} $\mathcal{P}:\mathbb{H}\to\mathcal{D}$ providing a \emph{holomorphic} map between the universal cover $\mathbb{H}$ of $C\simeq \overline{\mathbb{C}}-\{0,1,\infty\}$ and the so-called period domain $\mathcal{D}$ (see Chapiter 10 of Voisin's book \cite{Voisin}). By definition, the monodromy representation $\rho$ is computed from the period map $\mathcal{P}$, so that the desired ``Forni estimate'' follow from the Ahlfors-Schwartz-Pick lemma saying that the holomorphic map $\mathcal{P}$ is a contraction from the hyperbolic plane $\mathbb{H}$ equipped with the hyperbolic metric and the period domain $\mathcal{D}$ equipped with the natural metric (induced by the natural intersection form). In other words, the desired control of the monodromy matrices in terms of the hyperbolic distance follows from the fact that the holomorphic sectional curvatures of the period domain $\mathcal{D}$ are negative. 

This completes the proof of the theorem.
\end{proof}

\section*{Acknowledgments} The authors are thankful to Alex Furman for
suggesting the strategy  of the proof of Theorem~\ref{t.Lyapunov}, to Pascal
Hubert and Erwan Lanneau for sharing their insights on the geometry of
Prym Teichm\"uller curves of genus $4$, to Martin M\"oller and Jean-Christophe Yoccoz for useful exchanges around the Galois-theoretical simplicity criterion in the article~\cite{MMY}, and to Pascal Hubert and Julien Grivaux for a careful reading of earlier versions of this work. 

Research  of  the first author is partially supported  by
NSF grants DMS 0244542, DMS 0604251 and DMS 0905912.

The second author was partially supported by the French ANR grant ``GeoDyM'' (ANR-11-BS01-0004) and by the 
Balzan Research Project of J. Palis.

\end{document}